\documentclass[12pt,letterpaper]{amsart}
\usepackage{palatino, euler, epic, eepic, amssymb, xypic, floatflt, microtype, enumerate, stmaryrd}
\input xy
\xyoption{all}
 \newlength{\baseunit}               % the basic unit length
 \newcount{\numlines}                % depth of picture (in number of lines)
\usepackage{amsmath, amssymb, amscd, verbatim, xspace,amsthm}
\usepackage{latexsym, epsfig, color}

\newcommand\isom{\cong}

\newcommand\Spec{\operatorname{Spec}}

\newcommand\bq{\begin{equation}}
\newcommand\eq{\end{equation}}

\newtheorem{proposition}{Proposition}[section]
\newtheorem{theorem}[proposition]{Theorem}
\newtheorem{corollary}[proposition]{Corollary}

\newtheorem{lemma}[proposition]{Lemma}

\theoremstyle{definition}

\theoremstyle{remark}
\newtheorem{remark}[proposition]{Remark}
\usepackage{url}

\numberwithin{equation}{section}

\newcommand{\cut}[1]{}
\newcommand\hidden[1]{}

\newcommand{\cM}{\mathcal{M}}

\newcommand{\PP}{\mathbb{P}}

\newcommand{\setmin}{{\smallsetminus}}                                %
\newcommand{\Mp}{{\mathcal{M}^+}}                                     %
\newcommand{\Rel}{{\mathcal{R}}}                                      %
\newcommand{\dra}{\dashrightarrow}                                    %
\newcommand{\Ker}{\operatorname{Ker}}                                 %
\newcommand{\Img}{\operatorname{Img}}                                 %
\newcommand{\ZZ}{{\mathbb{Z}}}                                        %
\newcommand{\LL}{{\mathbb{L}}}                                        %
\newcommand{\cO}{{\mathcal O}}                                        %
\newcommand{\ch}{\operatorname{\tilde{c}_1}}                          %
\newcommand{\Schk}{\operatorname{Sch}_{k}}                            %
\newcommand{\Smk}{\operatorname{Sm}_{k}}                              %
\title{Descent for Algebraic Cobordism}

\author{Jos\'e Luis Gonz\'alez and Kalle Karu}
\address{J.L. Gonz\'alez,  Dept. of Mathematics, University of British Columbia,
  Vancouver, BC V6T1Z2, CANADA  \newline \indent
K. Karu,
Dept. of Mathematics, University of British Columbia, 
  Vancouver, BC V6T1Z2, CANADA} 
\email{jgonza@math.ubc.ca, karu@math.ubc.ca}
\thanks{This research was funded by NSERC Discovery and Accelerator grants.}

\begin{document}
\begin{abstract}
We prove the exactness of a descent sequence relating the algebraic cobordism groups of a scheme and its envelopes. Analogous sequences for Chow groups and K-theory were previously proved by Gillet. 

\end{abstract}
\maketitle
\setcounter{tocdepth}{1} % this just includes subsections

%\tableofcontents

%************************************************************************************************************************

%%%%%%%%%%%%%%%%%%%%%%%%%%%%%%%%%%%%%%%%%%%%%%%%%%%%%%%%%%%%%%%%%%%%%%%%%%%%%%%%%%%%%%%%%%%%%%%%%%%%%%%%%%%%%%%%%%%%%%%%%%%%%%%%
%%                                                                                                                            %%
%%                                                 SECTION 1.    INTRODUCTION                                                 %%
%%                                                                                                                            %%
%%%%%%%%%%%%%%%%%%%%%%%%%%%%%%%%%%%%%%%%%%%%%%%%%%%%%%%%%%%%%%%%%%%%%%%%%%%%%%%%%%%%%%%%%%%%%%%%%%%%%%%%%%%%%%%%%%%%%%%%%%%%%%%%

\section{Introduction}

We work in the category $\Schk$ of separated finite type schemes over a field $k$ of characteristic zero. Recall that a proper morphism of schemes $\pi:\tilde{X}\to X$ is an envelope if for every subvariety $V\subset X$, there exists a subvariety $\tilde{V}\subset \tilde{X}$, such that $\pi$ maps $\tilde{V}$ birationally onto $V$. 

Gillet in \cite{GilletSeq} proved that if $\pi:\tilde{X} \to X$ is an envelope, with $\pi$ projective, then the following sequence is exact:
\begin{equation} \label{eq-gillet}
 A_*(\tilde{X}\times_X\tilde{X}) \xrightarrow{{p_1}_*-{p_2}_*} A_*(\tilde{X}) \xrightarrow{\pi_*} A_*(X) \xrightarrow{} 0.
\end{equation}   
Here $A_*$ is either the Chow theory or the $K$-theory of coherent sheaves,  $p_i: \tilde{X}\times_X\tilde{X} \to \tilde{X}$ are the two projections and ${p_i}_*, \pi_*$ are the push-forward maps in either theory. 
The goal of this article is to prove the exactness of the analogous sequence in the cobordism theory $\Omega_*$  defined by Levine and Morel \cite{Levine-Morel}.

\begin{theorem} \label{thm-seq1} Let $\pi: \tilde{X}\to X$ be an envelope, with $\pi$ projective. Then the sequence
 \[ \Omega_*(\tilde{X}\times_X\tilde{X}) \xrightarrow{{p_1}_*-{p_2}_*} \Omega_*(\tilde{X}) \xrightarrow{\pi_*} \Omega_*(X) \xrightarrow{} 0\]
is exact.
\end{theorem}

This theorem has several applications. We mention here  two of them. The first application is in the relationship between algebraic cobordism $\Omega_*(X)$ and algebraic K-theory $G_0(X)$ (the Grothendieck group of the category of coherent sheaves on $X$). Levine and Morel in \cite{Levine-Morel} constructed a natural morphism 
\[ \Omega_*(X)\otimes_\LL \ZZ[\beta,\beta^{-1}] \longrightarrow G_0(X)[\beta,\beta^{-1}],\]
and proved it to be an isomorphism for smooth $X$. Dai in \cite{DaiPaper} extended this isomorphism to all schemes $X$ that can be embedded in a smooth scheme; this includes all quasiprojective schemes $X$. In \cite{ktheory} we build on Dai's work to prove this isomorphism for all schemes $X$ in $\Schk$. Since every scheme $X$ admits a quasiprojective envelope, Theorem~\ref{thm-seq1} implies that the cobordism theory $\Omega_*$ as a functor on $\Schk$ is determined by its restriction to the full subcategory of quasiprojective schemes. A similar statement for $G_0$ proved by Gillet \cite{GilletSeq} then reduces the problem to the case proved by Dai.

The second application of Theorem~\ref{thm-seq1} is in the study of operational bivariant theories of Fulton and MacPherson \cite{Fulton-MacPherson}. Kimura in \cite{Kimura} used the exact sequence (\ref{eq-gillet}) to give an inductive method for finding the operational Chow groups of singular varieties from the Chow groups of smooth varieties. His proof can be generalized to algebraic cobordism  and other homology theories. In \cite{bivariant} we study the operational bivariant theories associated to certain homology theories that have the exact descent sequence. As a special case, we describe the operational equivariant cobordism theory of toric varieties. This result is based on previous work by Payne \cite{Payne} and Krishna and Uma \cite{Krishna-Uma}.

%A note about terminology:  It would be more appropriate to call the homology theory $\Omega_*$ the algebraic bordism theory and the cohomology theory $\Omega^*$ the algebraic cobordism theory. However, we will simply call them both algebraic cobordism theories to be compatible with the terminology in \cite{Levine-Morel, Levine-Pandharipande}. We thank Shoji Yokura for drawing our attention to this.

%%%%%%%%%%%%%%%%%%%%%%%%%%%%%%%%%%%%%%%%%%%%%%%%%%%%%%%%%%%%%%%%%%%%%%%%%%%%%%%%%%%%%%%%%%%%%%%%%%%%%%%%%%%%%%%%%%%%%%%%%%%%%%%%
%%                                                                                                                            %%
%%                                      SECTION 2. An Overview of Algebraic Cobordism Theory                                  %%  
%%                                                                                                                            %%
%%%%%%%%%%%%%%%%%%%%%%%%%%%%%%%%%%%%%%%%%%%%%%%%%%%%%%%%%%%%%%%%%%%%%%%%%%%%%%%%%%%%%%%%%%%%%%%%%%%%%%%%%%%%%%%%%%%%%%%%%%%%%%%%

\section{An Overview of Algebraic Cobordism Theory}              \label{section.overview.cobordism}

Algebraic cobordism theory was defined by Levine and Morel in \cite{Levine-Morel}. Later, Levine and Pandharipande \cite{Levine-Pandharipande} found a simpler presentation of the cobordism groups. We will use the construction of Levine-Pandharipande as the definition, but refer to Levine-Morel for its properties.

Let $\Smk$ be the full subcategory of $\Schk$ whose objects are smooth quasiprojective schemes over $\Spec k$. By a smooth morphism we always mean a smooth and quasiprojective morphism.

For $X$ in $\Schk$, let $\cM(X)$ be the set of isomorphism classes of projective morphisms $f: Y\to X$ for $Y\in \Smk$. This set is a monoid under disjoint union of the domains; let $\Mp(X)$ be its group completion. The elements of $\Mp(X)$ are called cycles. The class of $f: Y\to X$ in $\Mp(X)$ is denoted by $[f: Y\to X]$. The group $\Mp(X)$ is free abelian, generated by the cycles $[f: Y\to X]$ where $Y$ is irreducible.

A double point degeneration is a morphism $\pi: Y\to \PP^1$, with $Y \in \Smk$ of pure dimension, such that $Y_\infty = \pi^{-1}(\infty)$ is a smooth divisor on $Y$ and $Y_0=\pi^{-1}(0)$ is a union $A\cup B$ of smooth divisors intersecting transversely along $D=A\cap B$.  Define $\PP_D = \PP(\cO_D(A)\oplus \cO_D)$, where $\cO_D(A)$ stands for $\cO_Y(A)|_D$. (Notice that $\PP(\cO_D(A)\oplus \cO_D) \isom \PP(\cO_D(B) \oplus \cO_D)$ because $\cO_D(A+B)\isom \cO_D$.)

Let $X\in \Schk$ and let $Y\in \Smk$ have pure dimension. Let $p_1, p_2$ be the two projections of $X\times \PP^1$.  A double point relation is defined by a projective morphism $\pi: Y\to X\times\PP^1$, such that $p_2\circ \pi: Y\to \PP^1$ is a double point degeneration. Let 
\[ [Y_\infty \to X], \quad [A\to X],\quad [B\to X], \quad [\PP_D \to X] \]
be the cycles obtained by composing with $p_1$. The double point relation is 
\[ [Y_\infty \to X] -[A\to X] - [B\to X] + [\PP_D\to X] \in \Mp(X). \]

Let $\Rel(X)$ be the subgroup of $\Mp(X)$ generated by all the double point relations. The cobordism group of $X$ is defined to be
\[ \Omega_*(X) = \Mp(X)/\Rel(X).\]
The group $\Mp(X)$ is graded so that $[f: Y\to X]$ lies in degree $\dim Y$ when $Y$ has pure dimension. Since double point relations are homogeneous, this grading gives a grading on $\Omega_*(X)$. We write $\Omega_n(X)$ for the degree $n$ part of  $\Omega_*(X)$. 

There is a functorial push-forward homomorphism $f_*: \Omega_*(X)\to \Omega_*(Z)$ for $f: X\to Z$ projective, and a functorial pull-back homomorphism $g^*: \Omega_*(Z)\to \Omega_{*+d}(X)$ for $g: X\to Z$ a  
smooth morphism of relative dimension $d$. These homomorphisms are both defined on the cycle level.  Levine and Morel also construct pull-backs along l.c.i. morphisms and, more generally, refined l.c.i. pullbacks. We will not need these pullbacks below.

The cobordism theory has exterior products 
\[ \Omega_*(X)\times \Omega_*(W) \longrightarrow \Omega_*(X\times W),\]
defined on the cycle level:
 \[ [Y\to X] \times [Z\to W] = [Y\times Z \to X\times W].\]
These exterior products turn $\Omega_*(\Spec k)$ into a graded ring and $\Omega_*(X)$ into a graded module over $\Omega_*(\Spec k)$. When $X$ is in $\Smk$, we denote by $1_X$ the class $[id_X: X\to X]$.

%%%%%%%%%%%%%%%%%%%%%%%%%%%%%%%%%%%%%%%%%%%%%%%%%%%%%%%%%%%%%%%%%%%%%%%%%%%%%%%%%%%%%%%%%%%%%%%%%%%%%%%%%%%%%%%%%%%%%%%%%%%%%%%%
%                 								                 FIRST CHERN CLASS OPERATORS                                                 %
%%%%%%%%%%%%%%%%%%%%%%%%%%%%%%%%%%%%%%%%%%%%%%%%%%%%%%%%%%%%%%%%%%%%%%%%%%%%%%%%%%%%%%%%%%%%%%%%%%%%%%%%%%%%%%%%%%%%%%%%%%%%%%%%

\subsection{First Chern Class Operators}

Algebraic cobordism is endowed with first Chern class operators 
\[ \ch(L): \Omega_*(X)\to \Omega_{*-1}(X),\]
associated to any line bundle $L$ on $X$.
This operator is also denoted by $\ch(\mathcal{L})$, where $\mathcal{L}$ is the invertible sheaf of sections of $L$. 
We recall some properties of these operators that are needed below.

A formal group law on a commutative ring $R$ is a power series $F_R(u,v)\in R\llbracket u,v\rrbracket $ satisfying
\begin{enumerate}[(a)]
\item $F_R(u,0) = F_R(0,u) = u$,
\item $F_R(u,v)=  F_R(v,u)$,
\item $F_R(F_R(u,v),w) = F_R(u,F_R(v,w))$.
\end{enumerate}
Thus 
\[ F_R(u,v) = u+v +\sum_{i,j>0} a_{i,j} u^i v^j,\]
where $a_{i,j}\in R$ satisfy $a_{i,j}=a_{j,i}$ and some additional relations coming from property (c). We think of $F_R$ as giving a formal addition
\[ u+_{F_R} v = F_R(u,v).\]
There exists a unique power series $\chi(u) \in R\llbracket u\rrbracket $ such that $F_R(u,\chi(u)) = 0$. Denote $[-1]_{F_R} u = \chi(u)$. Composing $F_R$ and $\chi$, we can form linear combinations 
\[ [n_1]_{F_R} u_1 +_{F_R}   [n_2]_{F_R} u_2 +_{F_R} \cdots +_{F_R} [n_r]_{F_R} u_r \in R\llbracket u_1,\ldots,u_r\rrbracket \]
for $n_i\in \ZZ$ and $u_i$ variables.

There exists a universal formal group law $F_\LL$, and its coefficient ring $\LL$ is called the \emph{Lazard ring}. This ring can be constructed as the quotient of the polynomial ring $\ZZ[A_{i,j}]_{i,j>0}$ by the relations imposed by the three axioms above. The images of the variables $A_{i,j}$ in the quotient ring are the coefficients $a_{i,j}$ of the formal group law $F_\LL$. The ring $\LL$ is graded, with $A_{i,j}$ having degree $i+j-1$. The power series $F_\LL(u,v)$ is then homogeneous of degree $-1$ if $u$ and $v$ both have degree $-1$. 

It is shown in \cite{Levine-Morel} that the graded group $\Omega_*(\Spec k)$ is isomorphic to $\LL$. The formal group law on $\LL$ describes the first Chern class operators of tensor products of line bundles (property (FGL)  below). 

We list three properties satisfied by the first Chern class operators $\ch(L): \Omega_*(Y) \to \Omega_{*-1}(Y)$ for $Y \in \Smk$ and $L$ a line bundle on $Y$:
\begin{itemize}
\item[(Dim)] For $L_1,\ldots,L_r$ line bundles on $Y$, $r>\dim Y$,
\[\ch(L_1)\circ\cdots\circ \ch(L_r) (1_Y) = 0.\]
\item[(Sect)] If $L$ is a line bundle on $Y$ and $s\in H^0(Y,L)$ is a section such that the zero subscheme $i:Z\hookrightarrow Y$ of $s$ is smooth, then
\[ \ch(L)(1_Y) = i_*(1_Z).\]
\item[(FGL)] For two line bundles $L$ and $M$ on $Y$,
\[ \ch(L\otimes M)(1_Y) = F_\LL(\ch(L), \ch(M))(1_Y).\]
\end{itemize}

In the terminology of \cite{Levine-Morel, Levine-Pandharipande} the three properties imply that $\Omega_*$ is an oriented Borel-Moore functor of geometric type. 

The first Chern class operators of two line bundles commute: $\ch(L)\circ\ch(M) = \ch(M)\circ \ch(L)$, and they are compatible with smooth (l.c.i.) pull-backs, projective push-forwards and exterior products.  The property (Sect) above implies that if $L$ is a trivial line bundle on $X$, then the first Chern class operator of $L$ is zero.

%%%%%%%%%%%%%%%%%%%%%%%%%%%%%%%%%%%%%%%%%%%%%%%%%%%%%%%%%%%%%%%%%%%%%%%%%%%%%%%%%%%%%%%%%%%%%%%%%%%%%%%%%%%%%%%%%%%%%%%%%%%%%%%%
%                 								                     	DIVISOR CLASSES                                                        %
%%%%%%%%%%%%%%%%%%%%%%%%%%%%%%%%%%%%%%%%%%%%%%%%%%%%%%%%%%%%%%%%%%%%%%%%%%%%%%%%%%%%%%%%%%%%%%%%%%%%%%%%%%%%%%%%%%%%%%%%%%%%%%%% 

\subsection{Divisor Classes} \label{sec-div-cl}

Recall that a divisor $D$ on a smooth scheme $Y \in \Schk$ has strict normal crossings (s.n.c.) if at every point $p\in Y$ there exists a system of regular parameters $y_1,\ldots,y_n$, such that  $D$ is defined by the equation $y_1^{m_1}\cdots y_n^{m_n}=0$ near $p$ for some integers $m_1, \ldots, m_n$.

Let $D = \sum_{i=1}^r n_i D_i$ be a nonzero s.n.c. divisor on a scheme $Y \in \Smk$, with $D_i$ irreducible. Let us recall the construction by Levine and Morel \cite{Levine-Morel} of the class $[D\to |D|] \in \Omega_*(|D|)$.  

Let 
\[ F^{n_1,\ldots,n_r}(u_1,\ldots,u_r) = [n_1]_{F_\LL} u_1 +_{F_\LL}   [n_2]_{F_\LL} u_2 +_{F_\LL} \cdots +_{F_\LL} [n_r]_{F_\LL} u_r \in \LL\llbracket u_1,\ldots,u_r\rrbracket .\]
We decompose this power series as 
\[ F^{n_1,\ldots,n_r}(u_1,\ldots,u_r) = \sum_J F_J^{n_1,\ldots,n_r}(u_1,\ldots,u_r) \prod_{i\in J} u_i,\]
where the sum runs over nonempty subsets $J\subset \{1,\ldots, r\}$. The power series $F_J^{n_1,\ldots,n_r}$ are such that $u_i$ does not divide any nonzero term in $F_J^{n_1,\ldots,n_r}$ if $i\notin J$. 

For $i=1,\ldots,r$, let $L_i=\cO_Y(D_i)$. If $J\subset \{1,\ldots,r\}$, let  $i^J: D^J=\cap_{i\in J} D_i \hookrightarrow |D|$, and $L_i^J =  L_i|_{D^J}$. The class $[D\to|D|]$ is defined in \cite{Levine-Morel} as
\begin{equation} \label{eq2}
  [D\to|D|] = \sum_J i_*^J F_J^{n_1,\ldots,n_r} (L_1^J, \ldots, L_r^J) (1_{D^J}),
\end{equation}
where the sum runs over nonempty subsets $J\subset \{1,\ldots,r\}$ and   
$F_J^{n_1,\ldots,n_r} (L_1^J, \ldots, L_r^J)$ is the power series $F_J^{n_1,\ldots,n_r}$ evaluated on the first Chern classes of $L_1^J, \ldots, L_r^J$.

When pushed forward to $Y$, the class $[D\to|D|]$ becomes equal to $\ch(\cO(D))(1_Y)$. To see this, compute,
\begin{alignat*}{2}
\ch(\cO(D))(1_Y)  &= F^{n_1,\ldots,n_r} (L_1, \ldots, L_r) (1_Y)\\
&= \sum_J F_J^{n_1,\ldots,n_r} (L_1, \ldots, L_r) \prod_{i\in J} \ch(L_i) (1_Y).
\end{alignat*}
Applying the property (Sect) repeatedly, we get
\[  \prod_{i\in J} \ch(L_i) (1_Y) = [ D^J \hookrightarrow Y].\]
Compatibility of first Chern class operators with pull-backs of line bundles then gives the desired divisor class formula pushed forward to $Y$.

We note that in the definition of divisor classes it is not necessary to assume that $D_i$ are irreducible. We may let them be smooth but possibly reducible divisors and then the same formula holds.

%%%%%%%%%%%%%%%%%%%%%%%%%%%%%%%%%%%%%%%%%%%%%%%%%%%%%%%%%%%%%%%%%%%%%%%%%%%%%%%%%%%%%%%%%%%%%%%%%%%%%%%%%%%%%%%%%%%%%%%%%%%%%%%%
%                 								                A PAIR OF SMOOTH DIVISORS                                                    %
%%%%%%%%%%%%%%%%%%%%%%%%%%%%%%%%%%%%%%%%%%%%%%%%%%%%%%%%%%%%%%%%%%%%%%%%%%%%%%%%%%%%%%%%%%%%%%%%%%%%%%%%%%%%%%%%%%%%%%%%%%%%%%%% 

\subsection{A Pair of Smooth Divisors.}
Let $A$ and $B$ be smooth divisors on $Y\in \Smk$ that intersect transversely along $D=A\cap B$. Consider
\[ [A+B \to |A\cup B|] = [A\to A\cup B]+ [B\to A \cup B]+ i_* F^{1,1}_{\{1,2\}}(\cO_D(A), \cO_D(B)) (1_D),\]
where $i: D\hookrightarrow |A\cup B|$ and $\cO_D(A), \cO_D(B)$ stand for $\cO_Y(A)|_D, \cO_Y(B)|_D$.  Let 
 \[ \PP_D = \PP(\cO_D(A)\oplus \cO_D).\]

\begin{lemma} \label{lem-pair}
With notation as above, let $\cO_Y(A+B+C)|_D = \cO_D$ for some divisor $C$ on $Y$. Then
\[ F^{1,1}_{\{1,2\}}(\cO_D(A), \cO_D(B)) (1_D) = -[\PP_D\to D] + \beta,\]
where 
\[ \beta = \sum_{i,j\geq 0, l> 0} b_{ijl} \ch(\cO_D(A))^i \ch(\cO_D(B))^j \ch(\cO_D(C))^l (1_D),\]
for some universally defined $b_{ijl}\in \LL$ that do not depend on $Y,A,B,C$.
\end{lemma}

\begin{proof} It is shown in the proof of \cite[Lemma~9]{Levine-Pandharipande} that
\[  -[\PP_D\to D]  = F^{1,1}_{\{1,2\}}(\cO_D(A), \cO_D(-A)) (1_D).\]
Substituting $\cO_D(B+C) = \cO_D(-A)$, we get
\begin{alignat*}{2} -[\PP_D\to D]  &= F^{1,1}_{\{1,2\}}(\cO_D(A), \cO_D(B+C)) (1_D) \\
&= \sum_{i,j>0} a_{ij} \ch(\cO_D(A))^{i-1} \ch(\cO_D(B+C))^{j-1}(1_D),
\end{alignat*}
with $a_{ij}$ the coefficients in the formal group law of $\LL$. 
To compute $\ch(\cO_D(B+C))$, we apply the formal group law again:
\[ \ch(\cO_D(B+C)) = \ch(\cO_D(B))+\ch(\cO_D(C))+ \sum_{i,j> 0} a_{ij} \ch(\cO_D(B))^i \ch(\cO_D(C))^j.\]
Now substituting this into the expression of $-[\PP_D\to D]$, the terms that involve only $\ch(\cO_D(A))$ and $\ch(\cO_D(B))$ give $F^{1,1}_{\{1,2\}}(\cO_D(A), \cO_D(B))(1_D)$, the remaining terms give the class $-\beta$.  
\end{proof}

%%%%%%%%%%%%%%%%%%%%%%%%%%%%%%%%%%%%%%%%%%%%%%%%%%%%%%%%%%%%%%%%%%%%%%%%%%%%%%%%%%%%%%%%%%%%%%%%%%%%%%%%%%%%%%%%%%%%%%%%%%%%%%%%
%%                                                                                                                            %%
%%                     SECTION 3.   Proof of the main theorem: Gillet Type Exact Sequence in the Cobordism Theory             %%
%%                                                                                                                            %%
%%%%%%%%%%%%%%%%%%%%%%%%%%%%%%%%%%%%%%%%%%%%%%%%%%%%%%%%%%%%%%%%%%%%%%%%%%%%%%%%%%%%%%%%%%%%%%%%%%%%%%%%%%%%%%%%%%%%%%%%%%%%%%%%

\section{Proof of the main theorem}     \label{section.gillet.cobordism}

We will now prove Theorem~\ref{thm-seq1}. Let us start with some notation. Recall that
\[ \Omega_*(X) = \Mp(X)/\Rel(X),\]
where $\Mp(X)$ is the group of cycles and $\Rel(X)$ is the subgroup of relations, generated by double point relations. We identify $\Omega_*(\Spec k)$ with the Lazard ring $\LL$. Then $\Omega_*(X)$ is an $\LL$-module.

Composition with $\pi$ gives the push-forward map $\pi_*: \Mp(\tilde{X})\to \Mp(X)$, taking $\Rel(\tilde{X})$ to $\Rel(X)$. The induced map $\Omega_*(\tilde{X}) \to \Omega_*(X)$ is also denoted $\pi_*$.  Define
\[ \Omega_*(X)_\pi = \pi_* \Mp(\tilde{X}) / \pi_* \Rel(\tilde{X}).\]
If $K_\pi = \Ker(\pi_*: \Mp(\tilde{X})\to \Mp(X))$, then
\begin{equation} \label{eq1}
 \Omega_*(X)_\pi \isom \Mp(\tilde{X}) /  (\Rel(\tilde{X})+K_\pi).
\end{equation}

There are natural maps of $\LL$-modules that factor $\pi_*$ as
\[ \Omega_*(\tilde{X}) \stackrel{\phi}{\longrightarrow}  \Omega_*(X)_\pi \stackrel{\psi}{\longrightarrow} \Omega_*(X).\]
The map $\phi$ is surjective by (\ref{eq1}). We claim that $\psi$ is also surjective. For every subvariety $Y\subset X$, choose a resolution of singularities given by a projective birational morphism $\tilde{Y}\to Y$. Then the cycles $[\tilde{Y}\to X]$ generate $\Omega_*(X)$ as an $\LL$-module \cite{Levine-Morel}. We can choose the resolutions so that $\tilde{Y}\to X$ factor through $\pi:\tilde{X}\to X$, hence these classes lie in $\pi_* \Mp(\tilde{X})$.

The following is the main ingredient for proving Theorem~\ref{thm-seq1}.

\begin{proposition} \label{prop-isom}
 The map $\psi:  \Omega_*(X)_\pi \to \Omega_*(X)$ is an isomorphism. 
\end{proposition}
 
{\em Proof of Theorem~\ref{thm-seq1}}. Let us assume Proposition~\ref{prop-isom} and prove Theorem~\ref{thm-seq1}. Clearly the sequence is a complex by functoriality of push-forward. Moreover, $\pi_* = \psi\circ\phi$ is surjective. We only need to prove exactness in the middle: $\Ker \pi_* \subset \Img ( {p_1}_*-{p_2}_* )$. By Proposition~\ref{prop-isom}, $\Ker \pi_* = \Ker \phi$, which by (\ref{eq1}) is the image of $K_\pi$ in   $\Omega_*(\tilde{X})$. Now $K_\pi$ is generated by cycles 
\[ [f:Y\to \tilde{X}] - [g:Y\to \tilde{X}],\]
where $\pi\circ f= \pi\circ g$. These generators lift to cycles 
\[ 
\pushQED{\qed} 
[(f,g): Y\to \tilde{X}\times_X\tilde{X}] \in \Omega_*(\tilde{X}\times_X\tilde{X}). \qedhere
\popQED
 \]

To prove Proposition~\ref{prop-isom}, we follow the argument in \cite[Chapter 6]{Levine-Morel} showing that $\Omega_*(X)_D \to \Omega_*(X)$ is an isomorphism. Here $\Omega_*(X)_D$ is the group defined by cycles and relations transverse to a divisor $D$. The proof has two steps:
\begin{enumerate}
 \item Define a distinguished lifting $\Mp(X) \xrightarrow{d} \Omega_*(X)_\pi$, such that the composition 
 \[ \Mp(\tilde{X}) \stackrel{\pi_*}{\longrightarrow} \Mp(X)  \stackrel{d}{\longrightarrow} \Omega_*(X)_\pi \]
 is the canonical homomorphism.
\item Show that $d$ maps $\Rel(X)$ to zero, hence it descends to $d: \Omega_*(X) \to \Omega_*(X)_\pi$, providing a left inverse to $\psi$ and proving that $\psi$ is injective.
\end{enumerate}

%%%%%%%%%%%%%%%%%%%%%%%%%%%%%%%%%%%%%%%%%%%%%%%%%%%%%%%%%%%%%%%%%%%%%%%%%%%%%%%%%%%%%%%%%%%%%%%%%%%%%%%%%%%%%%%%%%%%%%%%%%%%%%%%
%                 								              ELIMINATION OF INDETERMINACIES                                                  %
%%%%%%%%%%%%%%%%%%%%%%%%%%%%%%%%%%%%%%%%%%%%%%%%%%%%%%%%%%%%%%%%%%%%%%%%%%%%%%%%%%%%%%%%%%%%%%%%%%%%%%%%%%%%%%%%%%%%%%%%%%%%%%%%

\subsection{Elimination of Indeterminacies}

We will need to eliminate the indeterminacies of rational maps $Y\dra
\tilde{X}$.  We use the following well-known theorem.

\begin{theorem} (Hironaka \cite{Hironaka}) \label{thm-resolution} Let
$Y$ be a smooth variety, $D\subset Y$ a divisor with strict normal
crossings. Let $f: \tilde{Y}\to Y$ be a projective birational morphism,
$U\subset Y$ a nonempty  open set such that $f: f^{-1}(U)\to U$ is an
isomorphism. Then there exists a sequence of morphisms
  \[ Y=Y_1 \stackrel{g_1}{\longleftarrow} Y_2
\stackrel{g_2}{\longleftarrow} Y_3 \longleftarrow \ldots
\stackrel{g_{m-1}}{\longleftarrow} Y_m,\]
such that
\begin{enumerate}
  \item For each $i$ the map $g_i$ is the blowup of $Y_{i}$ along a
smooth center $C_i$, where $C_i$ lies over $Y\setmin U$ and intersects
the union of the pull-back of $D$ and the exceptional locus of $Y_i\to Y$
normally.
\item The rational map $Y_m\to Y\dra \tilde{Y}$ extends to a morphism
$Y_m\to \tilde{Y}$.
\end{enumerate}
\end{theorem}

Recall that if $Y$ is smooth and $D$ is an s.n.c. divisor on $Y$, then a
smooth subscheme $C\subset Y$ is said to intersect $D$ normally if at
every point $p\in Y$ we can choose a regular system of parameters
$y_1,\ldots, y_r$ so that $D$ is defined by $y_1^{n_1}\cdots
y_r^{n_r}=0$ for some $n_1,\ldots,n_r\in \ZZ$ and $C$ is defined by
vanishing of $y_{i_1},\ldots,y_{i_j}$ for some $i_1,\ldots,i_j$. If $D$
is an s.n.c. divisor and $C$ intersects it normally, then the blowup of
$Y$ along $C$ is smooth and the pull-back of $D$ together with the
exceptional divisor is again an s.n.c. divisor.

The proof of the above theorem is reduced to the problem of principalizing an
ideal sheaf as follows. One may assume that $f: \tilde{Y}\to Y$ is the
blowup of a coherent sheaf of ideals $I$ on $Y$ with co-support in
$Y\setmin U$. If the sequence of blowups $g: Y_m\to Y$ is such that
$g^*(I)$ is principal, then the birational map $Y_m\to \tilde{Y}$ is a
morphism.

We will apply Theorem~\ref{thm-resolution} in the following situation.

\begin{corollary}\label{cor-resolution}
Let $Y$ be a smooth variety and $D$ an s.n.c. divisor on $Y$. Let $\phi:
Y\to X$ be a proper morphism. Since $\pi:\tilde{X}\to X$ is an envelope,
there exists a subvariety $Z\subset \tilde{X}$ mapping birationally onto
$\phi(Y)\subset X$. Let $V\subset \phi(Y)$ be a nonempty open subset
such that $\pi|_{Z}^{-1}(V) \to V$ is an isomorphism, and let $U =
\phi^{-1}(V)\subset Y$.

Then there exists a sequence of blowups $g:Y_m\to Y$ of smooth centers
that lie over $Y\setmin U$ and intersect the inverse image of $D$
together with the exceptional locus normally, such that the composition
$Y_m\to Y\to X$ factors through $\tilde{X}$.
\end{corollary}

\begin{proof}
  Let $\tilde{Y}$ be the component of $Y \times_X Z$ that dominates $Y$.
Since $Z \to \phi(Y)$ is projective and birational, the projection $f:
\tilde{Y}\to Y$ is also projective and birational. Moreover, $f$ is an
isomorphism over $U$. We may now apply Theorem~\ref{thm-resolution}.
\end{proof}

The following result can be viewed as an embedded elimination of
indeterminacies.

\begin{corollary} \label{cor-res} Let $W$ be a smooth variety, $D, E$ be
effective divisors on $W$ such that $D+E$ has s.n.c., and $\phi:W\to X$ be a
proper morphism. Then there exists a birational morphism $g:\tilde{W}\to
W$, obtained by a sequence of blowups of smooth centers that lie over
$|D|$ and intersect the pull-back of $D+E$ together with the exceptional
locus normally, such that for every component $\tilde{D}_i$ of the
pull-back $\tilde{D}=g^*(D)$, the composition $\tilde{D}_i\to \tilde{W}
\to W\to X$ factors through $\tilde{X}\to X$.
\end{corollary}

\begin{proof}
Let $D_i$ be a component of $D$. Let $D'$ be the divisor
$D'=(D+E-qD_i)|_{D_i}$, where $q$ is the coefficient of $D_i$ in $D+E$.
We apply Corollary~\ref{cor-resolution}  to the map $D_i \to X$ and the
s.n.c. divisor $D'$ on $D_i$. The result is a sequence of blowups
$\tilde{D}_i\to D_i$ so that the composition $\tilde{D}_i\to D_i\to X$
factors through $\tilde{X}$. The centers of the blowups all lie over
$D_i\setmin U$, where $U=\phi^{-1}(V)$ for some nonempty open $V\subset
\phi(D_i)$.

Let us now perform the same sequence of blowups on $W$ (blow up $W$
along the same centers lying in $D_i$ and in the strict transforms of
$D_i$), to get $g: \tilde{W}\to W$. Then $\tilde{D}_i$ is isomorphic to
the strict transform of $D_i$ in $\tilde{W}$. Such blowups introduce new
components to the divisor $g^*(D)$. However, we claim that all these new
exceptional components have image in $X$ of smaller dimension than the
image of $D_i$. Indeed, by the choice of $U$, the centers of blowups lie
over the closed set $\phi(D_i)\setmin V \subset X$, and so do the
exceptional divisors.
Thus, by induction on the dimension of the
image in $X$, we can resolve the indeterminacies of all components of
(the pull-back of) $D$.
\end{proof}

%%%%%%%%%%%%%%%%%%%%%%%%%%%%%%%%%%%%%%%%%%%%%%%%%%%%%%%%%%%%%%%%%%%%%%%%%%%%%%%%%%%%%%%%%%%%%%%%%%%%%%%%%%%%%%%%%%%%%%%%%%%%%%%%
%                 								                  DISTINGUISHED LIFTINGS                                                     %
%%%%%%%%%%%%%%%%%%%%%%%%%%%%%%%%%%%%%%%%%%%%%%%%%%%%%%%%%%%%%%%%%%%%%%%%%%%%%%%%%%%%%%%%%%%%%%%%%%%%%%%%%%%%%%%%%%%%%%%%%%%%%%%%

\subsection{Distinguished Liftings}

Let $[Y\to X]$ be an element of $\cM(X)$, with $Y$ irreducible. We construct a lifting of this cycle to $\Omega_*(X)_\pi$.

Let $W=Y\times \PP^1$, $D=Y\times\{0\} \subset W$, and $f:W\to X$ be the composition of the projection to $Y$ and $Y\to X$. We apply Corollary~\ref{cor-res} to this situation (with $E=0$) to find a blowup $g:\tilde{W}\to W$. Let $\tilde{D} = g^*(D)$ be the pull-back of $D$, a possibly nonreduced s.n.c. divisor. 

Consider the class $[\tilde{D}\to |\tilde{D}|] \in \Omega_*(|\tilde{D}|)$.
Note that any map $Z\to |\tilde{D}|$ from an irreducible variety $Z$ has image in a component $\tilde{D}_i$ of $\tilde{D}$. Since $\tilde{D}_i\to X$ factors through $\tilde{X}\to X$, the composition $Z\to \tilde{D}_i \to X$ also factors. Similarly, an irreducible double point degeneration in $|\tilde{D}|$, when pushed forward to $X$, factors through $\tilde{X}$. It follows that the push-forward map $\Omega_*(|\tilde{D}|) \to \Omega_*(X)$ factors through $\Omega_*(X)_\pi$. We define a distinguished lifting of $[Y\to X]$ to be the image of  $[\tilde{D}\to  |\tilde{D}|]$ in $\Omega_*(X)_\pi$. Note that a distinguished lifting depends on the choice of the blowup $\tilde{W}\to W$, but not on the liftings $\tilde{D}_i\to \tilde{X}$ of $\tilde{D}_i\to X$.

%%%%%%%%%%%%%%%%%%%%%%%%%%%%%%%%%%%%%%%%%%%%%%%%%%%%%%%%%%%%%%%%%%%%%%%%%%%%%%%%%%%%%%%%%%%%%%%%%%%%%%%%%%%%%%%%%%%%%%%%%%%%%%%%
%                 								              PRODUCT OF DIVISOR CLASSES                                                     %
%%%%%%%%%%%%%%%%%%%%%%%%%%%%%%%%%%%%%%%%%%%%%%%%%%%%%%%%%%%%%%%%%%%%%%%%%%%%%%%%%%%%%%%%%%%%%%%%%%%%%%%%%%%%%%%%%%%%%%%%%%%%%%%%

\subsection{Product of Divisor Classes}

Let  $D$ and $E$ be effective divisors on a scheme $W \in \Smk$, such that $D+E$ has s.n.c. We define the class 
\[ [D\bullet E \to |D|\cap|E|] \in \Omega_*( |D|\cap|E|)\]
with the property that, when pushed forward to $W$, it becomes equal to 
\[ \ch(\cO_W(D))\circ \ch(\cO_W(E)) (1_W).\]

Let $D=\sum_i n_i D_i$ and $E=\sum_i p_i D_i$, where $D_i$, $i=1,\ldots,r$ are irreducible divisors. For $i=1,\ldots,r$, let $L_i=\cO_W(D_i)$. If $J\subset \{1,\ldots,r\}$ is such that $n_j\neq 0$ and $p_i\neq 0$ for some $i,j\in J$,  let  $i^J: D^J=\cap_{i\in J} D_i \hookrightarrow |D|\cap |E|$, and $L_i^J =  L_i|_{D^J}$.

Let the class $[D\bullet E \to |D|\cap|E|]$ be defined by the formula 
\[ \sum_{I, J} i_*^{I\cup J} F_J^{n_1,\ldots,n_r} (L_1^{I\cup J}, \ldots, L_r^{I\cup J}) F_I^{p_1,\ldots,p_r} (L_1^{I\cup J}, \ldots, L_r^{I\cup J}) \prod_{i\in I\cap J} \ch(L_i^{I\cup J})  (1_{D^{I\cup J}}).\]
Here the sum runs over pairs of nonempty subsets $I, J\subset \{1,\ldots,r\}$, such that $n_j \neq 0$ and $p_i\neq 0$ for all $j\in J$ and $i\in I$. 
As in the case of divisor classes, it is enough to assume that the divisors $D_i$ are smooth but not necessarily irreducible.

We claim that $[D\bullet E \to |D|\cap|E|]$, when pushed forward to $|D|$, becomes equal to  $\ch(\cO_W(E)|_{|D|}) [D\to |D|]$. To see this, we apply $\ch(\cO_W(E)|_{|D|})$ (which we shorten to  $\ch(\cO(E))$) to the definition of $[D\to |D|]$: 
\[ \ch(\cO(E)) [D\to |D|] =\\
 \sum_J i_*^J F_J^{n_1,\ldots,n_r} (L_1^J, \ldots, L_r^J) \ch(\cO(E)) (1_{D^J}).\]
We can now use the same divisor class formula to compute $\ch(\cO(E)) (1_{D^J})$:
\begin{alignat*}{2}
  \ch(\cO(E)) (1_{D^J}) &= F^{p_1,\ldots,p_r} (L_1^J, \ldots, L_r^J) (1_{D^J})\\
&= \sum_I F_I^{p_1,\ldots,p_r} (L_1^J, \ldots, L_r^J) \prod_{i\in I} \ch(L_i^J)(1_{D^J}).
\end{alignat*}
Applying the property (Sec), we get
\[ \prod_{i\in I} \ch(L_i^J)(1_{D^J}) = \prod_{i\in I\cap J} \ch(L_i^J) ([D^{I\cup J}\to D^J]).\]
Now putting the formulas back together and using the compatibility of the first Chern class operators with pull-backs proves the claim. 

Note that the above definition is symmetric in $E$ and $D$,
\[  [D\bullet E \to |D|\cap|E|] = [E\bullet D \to |D|\cap|E|].\]
This implies that, when pushed forward to $|D|\cup|E|$, the classes $\ch(\cO(E)) [D\to |D|]$ and $\ch(\cO(D)) [E\to |E|]$ become equal.

When $D$ is a smooth divisor that does not have common components with $E$, then $E'=E|_{|D|}$ is an s.n.c. divisor on $|D|$ and one has that 
\[ [D\bullet E \to |D|\cap|E|] = [E' \to |E'|].\]
To prove this, let $D=D_1$, $E=\sum_{i>1} p_i D_i$, and let the sum in the definition run over nonempty subsets $J\subset\{1\}$ and $I\subset\{2,\ldots,r\}$. Since $F_J^{1,0,\ldots,0} = 1$ for $J=\{1\}$, the sum simplifies to the expression defining the divisor class $[E' \to |E'|]$.

%%%%%%%%%%%%%%%%%%%%%%%%%%%%%%%%%%%%%%%%%%%%%%%%%%%%%%%%%%%%%%%%%%%%%%%%%%%%%%%%%%%%%%%%%%%%%%%%%%%%%%%%%%%%%%%%%%%%%%%%%%%%%%%%
%                 								                   	DOUBLE COBORDISMS                                                        %
%%%%%%%%%%%%%%%%%%%%%%%%%%%%%%%%%%%%%%%%%%%%%%%%%%%%%%%%%%%%%%%%%%%%%%%%%%%%%%%%%%%%%%%%%%%%%%%%%%%%%%%%%%%%%%%%%%%%%%%%%%%%%%%%

\subsection{Double Cobordisms}

Let $W$ be a scheme in $\Smk$,  $f: W\to \PP^1\times \PP^1$  a  morphism, $D=f^{*}(\PP^1\times\{0\})$, $E=f^{*}(\{0\}\times \PP^1)$. Assume that $D+E$ is an s.n.c. divisor on $W$. Write
\begin{alignat*}{3}
 D &=& \sum_i a_i D_i +\sum_i \alpha_i F_i,\\
E &= &\sum_i b_i E_i +\sum_i \beta_i F_i,
\end{alignat*}
where $F_i$ are the components of $D+E$ lying over $(0,0)\in\PP^1\times \PP^1$ and $D_i, E_i$ are the other components of $D$ and $E$. We may assume that $\alpha_i, \beta_i >0$.  

\begin{lemma} \label{lem-double}
With notation as above, let $D' = \sum_i a_i D_i$. Then the classes 
\[  [E\bullet D \to |E|\cap|D|]\qquad\text{and}\qquad  [E\bullet D' \to |E|\cap|D'|] \]
 become equal when pushed forward to $|E|$. 
\end{lemma}

\begin{proof}
 Pushed forward to $|E|$, the class  $[E\bullet D \to |E|\cap|D|]$ becomes equal to 
 \begin{gather*}
   \ch(\cO(D))[E\to|E|] = \ch(\cO(D')\otimes \cO(\sum_i \alpha_i F_i) )[E\to|E|] \\
= \ch(\cO(D'))[E\to|E|]  + \ch(\cO(\sum_i \alpha_i F_i) )[E\to|E|] \\
\qquad + \sum_{j,l\geq 1} a_{j,l}  \ch(\cO(D'))^j \ch(\cO(\sum_i \alpha_i F_i) )^l [E\to|E|],
\end{gather*}
where $a_{j,l}\in\LL$ are the coefficients of the formal group law. The first term in the sum gives $[E\bullet D' \to |E|\cap|D'|]$ pushed forward to $|E|$. It suffices to prove that the second term vanishes, because in that case the third term also vanishes. Now  $\ch(\cO(\sum_i \alpha_i F_i) )[E\to|E|]$ is the push-forward to $|E|$ of the class $[(\sum_i \alpha_i F_i) \bullet E \to  |\sum_i \alpha_i F_i| \cap |E|]$. If we instead push this class forward to $|\sum_i \alpha_i F_i|$, we get $\ch(\cO(E)) [\sum_i \alpha_i F_i \to  |\sum_i \alpha_i F_i| ] = 0$ because $\cO(E)$ is trivial on $|\sum_i \alpha_i F_i|$. Since $\beta_i>0$ for all $i$, the inclusion maps factor:
\[ |\sum_i \alpha_i F_i| \cap |E| \hookrightarrow  |\sum_i \alpha_i F_i| \hookrightarrow |E|,\]
and then the push-forward maps also factor.
\end{proof}

%%%%%%%%%%%%%%%%%%%%%%%%%%%%%%%%%%%%%%%%%%%%%%%%%%%%%%%%%%%%%%%%%%%%%%%%%%%%%%%%%%%%%%%%%%%%%%%%%%%%%%%%%%%%%%%%%%%%%%%%%%%%%%%%
%                 								              UNIQUENESS OF DISTINGUISHED LIFTINGS                                           %
%%%%%%%%%%%%%%%%%%%%%%%%%%%%%%%%%%%%%%%%%%%%%%%%%%%%%%%%%%%%%%%%%%%%%%%%%%%%%%%%%%%%%%%%%%%%%%%%%%%%%%%%%%%%%%%%%%%%%%%%%%%%%%%%

\subsection{Uniqueness of Distinguished Liftings}

\begin{lemma}
Let $[Y\to X]$ be a cycle in $\cM(X)$, with $Y$ irreducible, and consider two distinguished liftings of it defined by the images of  $[\tilde{D}_1\to |\tilde{D}_1|]$ and $[\tilde{D}_2\to |\tilde{D}_2|]$ in $\Omega_*(X)_\pi$, where
\begin{gather*}
g_1: \tilde{W}_1 \to Y\times \PP^1, \quad  \tilde{D}_1 = g_1^*(Y\times\{0\}),\\
g_2: \tilde{W}_2 \to Y\times \PP^1, \quad  \tilde{D}_2 = g_2^*(Y\times\{0\}).
\end{gather*}
Then the two distinguished liftings are equal in $\Omega_*(X)_\pi$.
\end{lemma}

\begin{proof} We may assume that the birational map $ \tilde{W}_1 \dra \tilde{W}_2$ is a morphism. (Otherwise find a third variety $\tilde{W}_3$ that maps to both of them.) It suffices to prove that the class $[\tilde{D}_1\to |\tilde{D}_1|]$, when pushed forward to $|\tilde{D}_2|$, becomes equal  to $[\tilde{D}_2\to |\tilde{D}_2|]$.

Since the morphism $\tilde{W}_1 \to \tilde{W}_2$ is proper and both varieties are quasi-projective, the morphism is projective. By the weak factorization theorem \cite{Wlodarczyk, AKMW}, we can factor the birational morphism $ \tilde{W}_1 \to \tilde{W}_2$ into a sequence of blowups and blowdowns along smooth centers. Moreover, the factorization can be chosen so that if $Z_{i+1}\to Z_{i}$ is one blowup of $C\subset Z_i$ in this factorization,  then the birational map $g_i: Z_i \dra \tilde{W}_2$ is a projective morphism, $g_i^*(\tilde{D}_2)$ is an s.n.c. divisor on $Z_i$, the center $C$ lies in the support of the divisor $g_i^*(\tilde{D}_2)$ and intersects it normally. 

We may thus assume that $ \tilde{W}_1 \to \tilde{W}_2$ is the blowup of $\tilde{W}_2$ along a smooth center $C\subset \tilde{W}_2$ that lies in the support of $\tilde{D}_2$ and intersects it normally. 
%Let $\tilde{W}_2 \to W$ be the blowup along a closed subscheme $Z \subset W$, supported on $Y\times\{0\} \subset W$.

Let $\tilde{V}_2 = \tilde{W}_2\times\PP^1$. Let  $\tilde{V}_1$ be the blowup of $\tilde{V}_2$ along $C\times\{0\} \subset \tilde{V}_2$. Let $f: \tilde{V}_1 \to \PP^1\times\PP^1$  be the projection. Consider the divisors $D=f^*(\PP^1\times\{0\})$ and $E=f^*(\{0\}\times \PP^1)$. Then $D+E$ is an s.n.c. divisor. Moreover, $D=D'+F$, where $F$ is the exceptional divisor of the blowup, lying over $(0,0)\in\PP^1\times\PP^1$, and $D'\isom \tilde{W}_1$. Since $D'$ is smooth, having no common component with $E$, and $E|_{D'} = \tilde{D}_1$, we get 
\[ [D'\bullet E\to |D'|\cap |E|] = [\tilde{D}_1\to |\tilde{D}_1|]. \]
Lemma~\ref{lem-double} implies that, when pushed forward to $|E|$, this class becomes $\ch(\cO(D))[E\to |E|]$, which itself is equal to  $[\tilde{D}_2\to |\tilde{D}_2|]$ pushed forward to $|E|$ by a section $s: |\tilde{D}_2|\to |E|$ of the projection $|E|\to  |\tilde{D}_2|$ . The two classes are equal when pushed forward to $|\tilde{D}_2|$.
\end{proof}

The previous lemma proves that distinguished liftings are unique. 
We extend the liftings of generators $[Y\to X]$ linearly to a group homomorphism $d: \Mp(X)\to \Omega_*(X)_\pi$. The distinguished lifting of the class $\pi_*[Y\to \tilde{X}] \in \Mp(X)$ is the class $\pi_*[Y\to \tilde{X}] \in \Omega_*(X)_\pi$, hence the composition
 \[ \Mp(\tilde{X}) \stackrel{\pi_*}{\longrightarrow} \Mp(X)  \stackrel{d}{\longrightarrow} \Omega_*(X)_\pi \]
is the canonical projection.

\begin{remark} \label{rem-dist-lift} The proof of the lemma shows that, to define the distinguished lifting of $[Y\to X]$, it is not necessary to require that $\tilde{W}$ is the blowup of $Y\times\PP^1$; we can allow both blowups and blowdowns along centers lying over $Y\times \{0\}$. More precisely, we need a proper birational map $g: \tilde{W}\dra Y\times \PP^1$ from a smooth quasi-projective variety $\tilde{W}$,  satisfying
\begin{enumerate}
\item The map $g$ is a regular isomorphism over $Y\times (\PP^1\setmin\{0\})$.
\item The composition with the second projection $ \tilde{W}\dra Y\times \PP^1 \to \PP^1$ is a morphism whose fiber over $0$ is a divisor $\tilde{D}$ with s.n.c. on  $ \tilde{W}$.
\item The rational map $\tilde{W}\dra Y\times \PP^1 \to X\times \PP^1$ extends to a projective morphism.
\item The morphism $\tilde{D}\to X$ factors through $\pi: \tilde{X}\to X$ on each component of $\tilde{D}$.
\end{enumerate}
Then the image of the class $[\tilde{D}\to |\tilde{D}|]$ in $\Omega_*(X)_\pi$ defines the distinguished lifting of $[Y\to X]$.  
\end{remark}

%%%%%%%%%%%%%%%%%%%%%%%%%%%%%%%%%%%%%%%%%%%%%%%%%%%%%%%%%%%%%%%%%%%%%%%%%%%%%%%%%%%%%%%%%%%%%%%%%%%%%%%%%%%%%%%%%%%%%%%%%%%%%%%%
%                 								   Completion of the Proof of Proposition~\ref{prop-isom}                                    %
%%%%%%%%%%%%%%%%%%%%%%%%%%%%%%%%%%%%%%%%%%%%%%%%%%%%%%%%%%%%%%%%%%%%%%%%%%%%%%%%%%%%%%%%%%%%%%%%%%%%%%%%%%%%%%%%%%%%%%%%%%%%%%%%

\subsection{Completion of the Proof of Proposition~\ref{prop-isom}}

It remains to prove that the distinguished lifting $d: \Mp(X)\to \Omega_*(X)_\pi$ maps $\Rel(X)$ to zero. Consider a double point degeneration $f:W\to \PP^1$, $W\to X$.  Let $W_\infty = f^{-1}(\infty)$ be a smooth fiber and  $W_0 = f^{-1}(0) = A \cup B$. Recall that the double point relation is
\[ [W_\infty\to X] - [A\to X] - [B\to X] +[\PP_{A\cap B}\to X],\]
where $\PP_{A\cap B}=  \PP(\cO_{A\cap B}(A) \oplus \cO_{A\cap B})$. Since $\Rel(X)$ is generated by the double point relations, it suffices to prove that 
\[  d[W_\infty\to X] - d[A\to X] -d [B\to X] +d[\PP_{A\cap B}\to X] = 0 .\]

Let $V=W\times \PP^1$. We blow up $V$ along smooth centers lying over $W\times \{0\}$ that intersect the pull-back of $W_0\times \PP^1+ W_\infty\times \PP^1+ W\times\{0\}$  normally.  Let the result be $\tilde{V}$, such that the map from the inverse image of $W\times \{0\}$ to $X$ lifts to $\tilde{X}$ on every irreducible component.  Let
\[ g: \tilde{V} \to W\times\PP^1 \to \PP^1\times\PP^1.\]
Define 
\[ E=g^*(\PP^1\times \{0\}),  \quad D_0=g^*(\{0\}\times \PP^1), \quad  D_\infty=g^*(\{\infty\}\times \PP^1).\]
 Let $D_0'$, $D_\infty'$ be the sums of components in $D_0$, $D_\infty$ that do not map to $(0,0)$ or $(\infty,0)$ in $\PP^1\times\PP^1$. Then $D_\infty'$ is the blowup of $W_\infty\times\PP^1$ along centers lying over $W_\infty\times \{0\}$.  Since $D'_\infty$ is smooth and has no component in common with $E$, it follows that  $[D_\infty'\bullet E \to |D_\infty'|\cap |E|]$ is the divisor class of $E|_{D'_\infty}$, which  gives the distinguished lifting of $[W_\infty\to X]$. 
 
 Similarly,  $D_0'$ is the blowup of $W_0\times\PP^1 = (A\cup B)\times \PP^1$ along centers lying over $W_0\times \{0\}$. The divisor $D_0'$ is a union $A'\cup B'$ of  two smooth divisors, the blowups of $A\times \PP^1$ and $B\times \PP^1$. The intersection of these divisors is a blowup of $(A\cap B)\times\PP^1$. We claim that when pushed forward, $[D_0'\bullet E \to |D_0'|\cap |E|]$ gives the class $d[A\to X] +d [B\to X] -d[\PP_{A\cap B}\to X]$. 
 
 Let $A'=D_1$, $B'=D_2$ and $E=\sum_{i>2} p_i D_i$. Then in the formula defining $[D_0'\bullet E \to |D_0'|\cap |E|]$ we may take the sum over nonempty subsets $J\subset\{1,2\}$ and $I\subset\{3,\ldots,r\}$. We divide the formula into three pieces corresponding to $J=\{1\}, J=\{2\}, J=\{1,2\}$:
 \begin{alignat}{2}  \label{3sums}
 [D_0'\bullet E \to |D_0'|\cap |E|]  
  & =   \sum_{I} i_*^{I\cup \{1\}} F_I^{p_3,\ldots,p_r} (L_3^{I\cup \{1\}}, \ldots, L_r^{I\cup \{1\}}) (1_{D^{I\cup \{1\}}}) \notag \\
  & +   \sum_{I} i_*^{I\cup \{2\}} F_I^{p_3,\ldots,p_r} (L_3^{I\cup \{2\}}, \ldots, L_r^{I\cup \{2\}}) (1_{D^{I\cup \{2\}}})\\
 & +  \sum_{I} i_*^{I\cup \{1,2\}} F_{\{1,2\}}^{1,1} (L_1^{I\cup \{1,2\}}, L_2^{I\cup \{1,2\}}) F_I^{p_3,\ldots,p_r} (L_3^{I\cup \{1,2\}}, \ldots, L_r^{I\cup \{1,2\}}) (1_{D^{I\cup \{1,2\}}}).\notag
 \end{alignat} 
 Here we used that $F^{1,1}_J = 1$ if $|J|=1$. In this triple sum, the first sum is the push-forward of the class $[E|_{A'}\to |E|\cap A']$, hence it gives the distinguished lifting of $[A\to X]$. Similarly, the second sum gives the distinguished lifting of $[B\to X]$.  
 
In the remainder of the proof we show that the third sum in Equation~(\ref{3sums}) gives the distinguished lifting of $-[\PP_{A\cap B}\to X]$. This is sufficient to finish the proof. 
 Indeed, when pushed forward to $|E|$, by Lemma~\ref{lem-double} the classes $[D_\infty'\bullet E \to |D_\infty'|\cap |E|]$ and $[D_0'\bullet E \to |D_0'|\cap |E|]$ are equal to $\ch(\cO_{|E|}(D_\infty))[E\to |E|] = \ch(\cO_{|E|}(D_0))[E\to |E|]$. Since in $|E|\to X$ every component lifts to $\tilde{X}$, the two classes are equal in $\Omega_*(X)_\pi$.

 Consider for each subset $I\subset \{3,\ldots,r\}$ the smooth divisors $A'|_{D^I}$ and $B'|_{D^I}$ on $D^I$, intersecting transversely along $D^{I\cup\{1,2\}}$. Define 
\[ \PP_{D^{I\cup\{1,2\}}} = \PP (\cO_{D^{I\cup\{1,2\}}}(A')\oplus \cO_{D^{I\cup\{1,2\}}}).\]
Using the same formula for $I=\emptyset$ defines $\PP_{D^{\{1,2\}}} = \PP_{A'\cap B'}$. Then clearly
\[ \PP_{D^{I\cup\{1,2\}}} =  \PP_{A'\cap B'} \times_{A'\cap B'} D^{I\cup\{1,2\}}.\]
Let $D_0 = A' + B' + C'$, where $C'=\sum_{i>2} c_i D_i$ is a divisor lying over $(0,0)\in\PP^1\times\PP^1$. Then by Lemma~\ref{lem-pair}, 
\begin{equation} \label{eq-pp}
F_{\{1,2\}}^{1,1} (L_1^{I\cup \{1,2\}}, L_2^{I\cup \{1,2\}}) (1_{D^{I\cup\{1,2\}}}) = -[\PP_{D^{I\cup\{1,2\}}} \to D^{I\cup\{1,2\}}] + \beta_I,
\end{equation}
where 
\[ \beta_I = \sum_{i,j\geq 0, l> 0} b_{ijl} \ch(\cO_{D^{I\cup\{1,2\}}}(A'))^i \ch(\cO_{D^{I\cup\{1,2\}}}(B'))^j \ch(\cO_{D^{I\cup\{1,2\}}}(C'))^l (1_{D^{I\cup\{1,2\}}}),\]
and the coefficients $b_{ijl}\in\LL$ are independent of $I$.

Consider the morphism $\PP_{A'\cap B'} \to A'\cap B' \to  \PP^1$. The projective bundles  $\PP_{A'\cap B'}$ and $\PP_{A\cap B}\times \PP^1$ are isomorphic over $\PP^1\setmin\{0\}$. Moreover, 
the induced birational map $\PP_{A'\cap B'} \to \PP_{A\cap B}\times \PP^1$ satisfies the conditions of Remark~\ref{rem-dist-lift}, hence  $\PP_{A'\cap B'}$ can be used to define the distinguished lifting of $[\PP_{A\cap B}\to X]$.  This distinguished lifting (with minus sign) is obtained by substituting (\ref{eq-pp}) into the third sum of Equation~(\ref{3sums}) and setting all $\beta_I=0$:
\[ \sum_{I} i_*^{I\cup \{1,2\}} F_I^{p_3,\ldots,p_r} (L_3^{I\cup \{1,2\}}, \ldots, L_r^{I\cup \{1,2\}}) (-[\PP_{D^{I\cup\{1,2\}}} \to D^{I\cup\{1,2\}}]).\]
Now it suffices to prove that the sum of the terms in Equation~(\ref{3sums}) involving $\beta_I$ vanishes when pushed forward to $|E|$. Then it also vanishes in $\Omega_*(X)_\pi$.

We claim that there exists a class $\alpha\in\Omega_*(|E|)$, such that for each $I$, the class $\beta_I$ pushed forward to $|E|$ is equal to 
\[ \left(\prod_{i\in I} \ch(\cO_{|E|}(D_i))\right)(\alpha).\]
 For this note that the first Chern class operators on $\Omega_*(|E|)$ commute in the following sense. For each $i,j>2$, when pushed forward to $|E|$, the classes
 \[  \ch(\cO_{D_j}(D_i) )(1_{D_j}) \ \ \textnormal{and} \ \ \ch(\cO_{D_i}(D_j)) (1_{D_i})\]
 become equal. Now consider $\ch(\cO_{D^{I\cup\{1,2\}}}(C')) (1_{D^{I\cup\{1,2\}}})$. Using the formal group law, we can express $\ch(\cO_{D^{I\cup\{1,2\}}}(C'))$ using $\ch(\cO_{D^{I\cup\{1,2\}}}(D_i))$ for $i>2$ and by the commutativity property, $\ch(\cO_{D^{I\cup\{1,2\}}}(C')) (1_{D^{I\cup\{1,2\}}})$ when pushed forward to $|E|$ becomes equal to 
\[ \left(\prod_{i\in I} \ch(\cO_{|E|}(D_i))\right)(\alpha')\]
for some $\alpha'\in \Omega_*(|E|)$, which is independent of $I$. Now set 
\[ \alpha = \sum_{i,j\geq 0, l> 0} b_{ijl} \ch(\cO_{|E|}(A'))^i \ch(\cO_{|E|}(B'))^j \ch(\cO_{|E|}(C'))^{l-1} (\alpha').\]
Then clearly this $\alpha$ satisfies the required property.

The sum of terms in Equation~(\ref{3sums}) involving $\beta_I$, when pushed forward to $|E|$, becomes equal to 
 \[  F^{p_3,\ldots,p_r} ( \cO_{|E|}(D_3),  \ldots, \cO_{|E|}(D_r)) (\alpha) = \ch(\cO_{|E|}(E))(\alpha).\]
 Since $\cO_{|E|}(E)$ is trivial, this class vanishes.
\qed

\bibliographystyle{plain}
\bibliography{cobordismbib}

\end{document}